\documentclass[11pt]{amsart}

\makeatletter
\@namedef{subjclassname@2020}{%
  \textup{2020} Mathematics Subject Classification}
\makeatother

\usepackage{amsmath, amsthm, amsfonts, amssymb,  mathrsfs, graphicx, lscape}
\usepackage{enumitem}
\usepackage[all]{xy}
\usepackage[usenames, dvipsnames]{color}
\usepackage[margin=1in]{geometry} 
\usepackage{tikz}
\usetikzlibrary{matrix,arrows}
\usepackage{tikz-cd}

\numberwithin{equation}{section}
\newtheorem{theorem}[equation]{Theorem}

\newtheorem{proposition}[equation]{Proposition}
\newtheorem{lemma}[equation]{Lemma}
\newtheorem{corollary}[equation]{Corollary}

\theoremstyle{definition}
\newtheorem{rmk}[equation]{Remark}
\newenvironment{remark}[1][]{\begin{rmk}[#1] \pushQED{\qed}}{\popQED \end{rmk}}
\newtheorem{eg}[equation]{Example}
\newenvironment{example}[1][]{\begin{eg}[#1] \pushQED{\qed}}{\popQED \end{eg}}
\newtheorem{defn}[equation]{Definition}

\newtheorem{ques}[equation]{Question}

\usepackage[in]{fullpage}
\usepackage[colorlinks=true, pdfstartview=FitV, linkcolor=black, citecolor=black, urlcolor=black]{hyperref}

\usepackage{ stmaryrd }
\usepackage[normalem]{ulem}

\DeclareMathOperator{\im}{image} 
\DeclareMathOperator{\codim}{codim}
\DeclareMathOperator{\coker}{coker}
\newcommand{\Hom}{\operatorname{Hom}}

\DeclareMathOperator{\rank}{rank}
\DeclareMathOperator{\Ext}{Ext}

\DeclareMathOperator{\Sym}{Sym}

\DeclareMathOperator{\gr}{gr}

\DeclareMathOperator{\rep}{rep}

\newcommand{\GL}{\operatorname{GL}}
\DeclareMathOperator{\cha}{char}

\newcommand{\op}{\operatorname}
\newcommand{\oo}{\otimes}
\newcommand{\la}{\lambda}
\newcommand{\al}{\alpha}
\newcommand{\mc}{\mathcal}
\newcommand{\ol}{\overline}
\newcommand{\V}{\mathcal{V}}
\newcommand{\bb}{\mathbb}
\newcommand{\St}{\operatorname{St}}

\newcommand{\bd}{\mathbf{d}}
\newcommand{\za}{\ensuremath{\alpha}}
\newcommand{\br}{\mathbf{r}}
\newcommand{\bs}{\mathbf{s}}

\newcommand{\kk}{\Bbbk}

\begin{document}
\title{On the collapsing of homogeneous bundles in arbitrary characteristic}

\author{Andr\'as Cristian L\H{o}rincz}
\address{Humboldt--Universit\"at zu Berlin, Institut f\"ur Mathematik, Berlin, Germany}
\email[Andr\'as Cristian L\H{o}rincz]{lorincza@hu-berlin.de}

\begin{abstract}
We study the geometry of equivariant, proper maps from homogeneous
bundles $G\times_P V$ over flag varieties $G/P$ to representations of $G$, called collapsing maps. Kempf showed that, provided the bundle is completely reducible, the image $G\cdot V$ of a collapsing map has rational singularities in characteristic zero. We extend this result to positive characteristic and show that for the analogous bundles the saturation $G\cdot V$ is strongly $F$-regular if its coordinate ring has a good filtration. We further show that in this case the images of collapsing maps of homogeneous bundles restricted to Schubert varieties are $F$-rational in positive characteristic, and have rational singularities in characteristic zero. We provide results on the singularities and defining equations of saturations $G\cdot X$ for $P$-stable closed subvarieties $X\subset V$. We give criteria for the existence of good filtrations for the coordinate ring of $G\cdot X$. 

Our results give a uniform, characteristic-free approach for the study of the geometry of a number of important varieties: multicones over Schubert varieties, determinantal varieties in the space of matrices, symmetric matrices, skew-symmetric matrices, and certain matrix Schubert varieties therein, representation varieties of radical square zero algebras (e.g.\ varieties of complexes), subspace varieties, higher rank varieties, etc.

\vspace{0.3in}

\begin{center}
\textbf{\normalsize{\uppercase{Sur l'effondrement des fibrés homogènes en caractéristique arbitraire
}}}
\end{center}

\vspace{0.2in}

\noindent \textsc{R\'esum\'e}. On étudie la géométrie des applications propres équivariantes de fibrés homogènes $G\times_P V$ sur les variétés de  drapeaux  $G/P$ dans les représentations de $G$, appelées applications d'effondrement. Kempf a montré que lorsque le fibré est complètement réductible, l'image $G\cdot V$ d'une application d'effondrement a des singularités rationelles en caractéristique zéro.  On étend ce résultat à la caractéristique positive et on montre que pour les fibrés analogues  la saturation $G\cdot V$ est fortement $F$-régulière si son anneau des coordonnées a une bonne filtration. De plus, on montre que dans ce cas les images des applications d'effondrement de fibrés homogènes restreintes aux variétés de Schubert sont $F$-rationelles en caractéristique positive, et ont des singularités rationelles en caractéristique zéro. On obtient des résultats sur les singularités et  les équations qui
définissent les saturations $G\cdot X$ pour les sous-variétés $X\subset V$ fermés
$P$-stables.
On donne un critère pour l'existence de bonnes filtrations pour l'anneau des coordonnées de  $G\cdot X$. 

Nos résultats fournissent une approche uniforme et
indépendante de la caractéristique,  à l'étude de la géométrie de nombreuses
variétés  importantes:  multicônes sur les variétés de Schubert, variétés
déterminantales dans l'espace de matrices, matrices symétriques, matrices
antisymétriques et certaines variétés de Schubert  de matrices, variétés
de représentations des algèbres dont le carré du radical est zéro (par ex.\
variétés de complexes), variétés de sous-espaces, variétés de rang
supérieur, etc.
\end{abstract}



\maketitle

\setcounter{tocdepth}{2}


\newpage

\section{Introduction}

Let $G$ be a connected reductive group over an algebraically closed field $\kk$. Consider a parabolic subgroup $P$ of $G$, and let $W$ be a $G$-module and $V\subset W$ a $P$-stable submodule. The saturation $G\cdot V \subset W$ is the image of the homogeneous vector bundle $G \times_P V$ under the proper \lq\lq collapsing map" \, $G \times_P V \to W$\, induced by the action of $G$ on $W$. 

Many remarkable varieties can be realized through such collapsing of bundles for various choices of $G$, $P$, $W$, $V$ (cf.\ Section \ref{sec:applications}; for more such examples, see \cite{weymanbook}). Generally, the study of their geometry has been undertaken on case-by-case basis. An exception is the seminal work \cite{Kempf76}, where it is shown that in characteristic zero $G\cdot V$ has rational singularities whenever the unipotent radical $U(P)$ of $P$ acts trivially on $V$ (see also \cite{Kempf86}). Further, in this case the singularities of $G\cdot X$ are shown to be well-behaved for a closed $P$-stable subvariety $X \subset V$ \cite[Proposition 1 and Theorem 3]{Kempf76}.

In this paper, we generalize and extend the scope of Kempf's results along several directions. In particular, we give characteristic-free strengthenings of the statements above, under the presence of good filtrations as initiated by Donkin \cite{donkin},  \cite{donkinconj}. We say that a $G$-variety $Z$ is good, if $\kk[Z]$ has a good filtration (see Section \ref{sec:goodfil}). We point out to the reader that all good-related properties hold automatically when $\cha \kk = 0$, and our results below are new in this case as well (with the exception of Theorem \ref{thm:introgood}). 

Let $B\subset P$ a Borel subgroup of $G$ and $T\subset B$ a maximal torus. We denote the set of dominant weights of $G$ by $X(T)_+$. For $\la \in X(T)_+$ we let $\Delta_G(\la)$ denote the corresponding Weyl module (see Section \ref{sec:cohomo}). We consider the Levi decomposition $P= L \ltimes U(P)$ with $L$ reductive. Pick any  $\la_1, \la_2,\dots, \la_n \in X(T)_+$, and for the rest of the introduction fix
\begin{equation}\label{eq:mainsetup}
W= \bigoplus_{i=1}^{n} \Delta_{G}(\la_i) \quad \mbox{and} \quad V=\bigoplus_{i=1}^{n} \Delta_{L}(\la_i).
\end{equation}
We have a natural inclusion $V\subseteq W^{U(P)}$, with equality if $\cha \kk = 0$ (when the bundle is completely reducible \cite{Kempf76}). While the examples in Section \ref{sec:applications} fit into the setup (\ref{eq:mainsetup}), we note that in Section \ref{sec:mainresults} we develop the results in a more general setting (see (\ref{eq:hsplit})).

\begin{theorem}\label{thm:introsat}
Let $X \subset V$ be an $L$-submodule such that $G\cdot X$ is good. Then $G\cdot X$ is strongly $F$-regular when $\cha \kk >0$ (resp.\ is of strongly $F$-regular type when $\cha \kk =0$).
\end{theorem}

This illustrates that good filtrations are responsible for the geometric behavior of saturations in positive characteristic, a phenomenon that is apparent in invariant theory as well \cite{hashiinv}, \cite{hashiunip}. Example \ref{ex:counter} demonstrates that this assumption cannot be dropped.

The following is our main criterion for the existence of good filtrations (for the definition of good pairs, see Section \ref{sec:goodfil}).

\begin{theorem}\label{thm:introgood}
Assume that $W$ is good, and that $(V,X)$ is a good pair for some closed $L$-variety $X\subset V$. Then $(\,W\,, \,\, G\cdot X)$ is a good pair of $G$-varieties.
\end{theorem}

In particular, this implies that $G\cdot V$ is good whenever $\cha \kk > \max \{ \dim \Delta_{G}(\la_i) \,| \, 1\leq i \leq n\}$. However, in concrete situations the bound on $\cha \kk$ can be further improved significantly (cf. Sections \ref{subsec:det}, \ref{sec:radsquare}). See Theorem \ref{thm:good} for other criteria in this direction.

We extend the collapsing method to various relative settings, thus greatly increasing its versatility. These include restrictions to Schubert varieties or multiplicity-free subvarieties of flag varieties (for the latter, see Corollary \ref{cor:multfree}). Below $\mc{W}$ denotes the Weyl group of $G$.

\begin{theorem}\label{thm:intromain}
Consider a closed $L$-variety $X\subset V$ and assume that $G\cdot X$ is good. For any $w\in \mc{W}$, we have:
\begin{enumerate}
\item $\ol{BwX}$ is normal if and only if $X$ is so.
\item If $\cha \kk = 0$, then $\ol{BwX}$ has rational singularities if and only if so does $X$.
\item If $\cha \kk> 0$ and $X$ is an $L$-submodule of $V$, then $\ol{BwX}$ is $F$-rational.
\end{enumerate}
\end{theorem}

Note that when $w$ is the longest element in $\mc{W}$, we have $\ol{BwX}=G\cdot X$.

\smallskip

Frequently (e.g.\ when $G\cdot X$ is a spherical variety), the varieties $\ol{BwX}$ are orbit closures under the action of the Borel subgroup $B$ (see Section \ref{sec:applications}). The singularities of such varieties have been investigated mostly in the spherical case (e.g. \cite{projnormschub}, \cite{brispher}, \cite{brithom}), but they are not well understood \cite[Comments 4.4.4]{perrin}. Theorem \ref{thm:intromain} is one of the first of its kind at this level of generality, applicable equally in non-spherical situations as well.

When $P$ is itself a Borel subgroup, we sharpen some results on singularities (see Corollary \ref{cor:borel}), extending the case of multicones over Schubert varieties \cite{multicone}, \cite{hashischub}. 

Next, we provide a relative result on the defining ideals of saturations $G\cdot X$. For this, we introduce the notion of good generators of an ideal, see Definition \ref{def:goodgen}.
\begin{theorem}\label{thm:introdefi}
Let $(V,X)$ be a good pair with $G\cdot V$ good, and denote by $I_X \subset \kk[V]$ the defining ideal of $X \subset V$. Let $M$ be the span of a set of good generators of $I_X$ and take a basis $\mc{P}'$ of the $G$-module $H^0(G/P, \mc{V}(M)) \, \subset \kk[G\cdot V]$. Consider:
\begin{enumerate}
\item A set of generators $\mc{P}_{G\cdot V}$ of the defining ideal $I_{G\cdot V} \subset \kk[W]$ of $G\cdot V$;
\item A lift $\tilde{\mc{P}'} \subset \kk[W]$ of the set $\mc{P}' \subset \, \kk[W]/I_{G\cdot V}$.
\end{enumerate}
Then the defining ideal of $G \cdot X$ in $\kk[W]$ is generated by the set $\mc{P}_{G\cdot V} \, \cup \, \tilde{\mc{P}'}$. 
\end{theorem} 

In Theorem \ref{thm:defi} we give a version of the above that yields good defining equations, which we use to readily find (good) defining equations for the examples in Sections \ref{subsec:det} and \ref{sec:radsquare}. 

Saturations of the type $G\cdot V$ appear in various forms throughout the existing literature, and a range of techniques have been developed to better understand their geometry. Applying the results above in the special case of radical square zero algebras (see Section \ref{sec:radsquare}), we simultaneously sharpen and generalize the main results in \cite{Kempfcomp}, \cite{DecoStrick}, \cite{Strick1}, \cite{brioncomp}, 
\cite{Strick2}, \cite{MT1}, \cite{MT2} that concern the singularities and defining equations of the Buchsbaum--Eisenbud varieties of complexes as well as varieties of complexes of other type. In addition, we obtain that certain $B$-orbit closures in varieties of complexes are $F$-rational when $\cha \kk >0$ (resp.\ have rational singularities when $\cha \kk=0$).

Our results provide a general method for the investigation of the geometry of parabolically induced orbit closures in a representation $W$ of a reductive group $G$. Namely, for any choice of a parabolic $P \subset G$, we can take the representation $V$ of the smaller reductive group $L$ as in (\ref{eq:mainsetup}) with trivial $U(P)$-action; choosing an $L$-orbit closure $X=\ol{Lx}$ (for any $x\in V$), saturation gives a $G$-orbit closure $G\cdot X = \ol{Gx} \subset G\cdot V \subset W$. By considering all such possible choices, we obtain a large set of $G$-orbit closures in $W$ whose singularities and defining equations are inherited from the smaller ones according to the results above.

The Cohen--Macaulay property for collapsing of bundles in positive characteristic is a consequence of the study of their $F$-singularities. This relies on techniques from tight closure theory that was developed by Hochster and Huneke \cite{hh1}, \cite{hh}. In Section \ref{subsec:ample}, we translate this property into Griffiths-type vanishing results for the cohomology of such bundles on Schubert varieties in positive characteristic, extending the classical Kodaira-type vanishing results for line bundles \cite{schubsplit}, \cite{schubfreg}, \cite{smith2}.

\subsection*{Acknowledgments}
The author would like to express his gratitude to Ryan Kinser for his valuable comments and suggestions on this work.


\section{Preliminaries}

We work over an algebraically closed field $\kk$ of arbitrary characteristic (see Remark \ref{rem:field}).

An action of an algebraic group $G$ on an algebraic variety $X$ is always assumed to be algebraic, so that the map
$G \times X \to X$ is a morphism of algebraic varieties. We call a (possibly infinite-dimensional) vector space $V$ a rational $G$-module, if $V$ is equipped with a linear action of $G$, such that every $v\in V$ is contained in a finite-dimensional $G$-stable subspace on which $G$ acts algebraically. All modules considered are assumed to be rational of countable dimension.

Unless otherwise stated, throughout a ring or algebra is commutative, finitely generated over $\kk$ with a multiplicative identity.

\subsection{Reductive groups}
Let $G$ be a connected reductive group over $\kk$, $B$ a Borel subgroup and $U$ its unipotent radical. We fix a maximal torus $T \subset B$, and denote by $X(T)$ its group of characters. We denote by $\langle \cdot, \cdot \rangle$ the standard pairing between $X(T)$ and the group of cocharacters. Let $\Phi \subset X(T)$ denote the set of roots and $\Phi_+ \subset \Phi$ the set of positive roots with respect to the choice of $B$. We denote by $\rho$ the half sum of all the positive roots. The set of simple roots in $\Phi_+$ is denoted by $S$. We let $\mc{W}=N(T)/T$ be the Weyl group of $G$, and $w_0 \in W$ its longest element.

For $I\subset S$, consider the standard parabolic subgroup $P:=P_I \subset G$. We have a Levi decomposition $P_I= L_I \ltimes U_I$, where $U_I$ is the unipotent radical of $P$ and $L:=L_I$ is reductive. Let $\mc{W}_I$ be the subgroup generated by the reflections $s_\al$ with $\al \in I$, and $w_I$ the longest element in $\mc{W}_I$. We choose the set $\mc{W}^I$ of representatives of the cosets of $\mc{W}/\mc{W}_I$ as
\begin{equation}\label{eq:weylrep}
\mc{W}^I=\{w\in \mc{W} \, | \, w(\al) \in \Phi_+, \mbox{ for all } \al \in I\}.
\end{equation}
We have the Bruhat decomposition of $G$ into $B\times P$-orbits (see \cite[Section II.13]{jantzen}):
\[G  \, = \, \bigcup_{w \in \mc{W}^I} BwP.\]
For $w \in \mc{W}^I$, we put $U(w):=U \cap w U^- w^{-1}$, where $U^-$ is the opposite unipotent radical. The multiplication map induces an isomorphism of $U(w)$-varieties (see \cite[Section II.13.8]{jantzen})
\begin{equation}\label{eq:opencell}
U(w) \times P \xrightarrow{\cong} BwP, \qquad (u,p) \, \mapsto \, u w p.
\end{equation}
We denote by $X(w)_P$ the Schubert variety that is the image of $\ol{BwP}$ under the locally trivial projection $G\to G/P$. For $P=B$, we write $X(w):=X(w)_B$.

\subsection{Cohomology of homogeneous bundles}\label{sec:cohomo}

For any representation $M$ of $P$, we denote by $\mc{V}(M)$ the sheaf of sections of the homogeneous vector bundle $G\times_P M$. For $\la \in X(T)$, we put $\mc{L}(\la):=\mc{V}( \kk_{-\la})$, where $\kk_{-\la}$ is the $1$-dimensional representation of $B$.

A weight $\la \in X(T)$ is dominant if $\langle \la , \al^\vee \rangle \geq 0$, for all simple roots $\al \in S$. The set of dominant weights is denoted by $X(T)_+$. For $\la \in X(T)_+$, we call the space of sections
\[\nabla_G(\la):=H^0(G/B, \mc{L}(\la)),\]
a \emph{dual Weyl module}. It has lowest weight $-\la$ and highest weight $-w_0 \cdot \la$. The module $\Delta_G(\la) = \nabla_G(\la)^*$ is called a \emph{Weyl module}, that has a non-zero highest weight vector of weight $\la$, and this generates $\Delta_G(\la)$ as a $G$-module. It is known that $\nabla_G(\la)$ has a unique simple submodule, of highest weight $-w_0 \cdot \la$. 

When $\cha \kk = p >0$ and $e\geq 1$ is an integer such that $(p^e-1)\rho$ is a weight of $G$, we denote by $\St_e = \nabla_G((p^e-1)\rho)$ the $e^{\op{th}}$ Steinberg module, and put $\St:=\St_1$. The assumption is superficial as we can always replace $G$ by $\op{rad} G \times \tilde{G}$, where $\op{rad} G$ denotes the radical of $G$ and $\tilde{G}$ the universal cover of $[G,G]$, and $(p^e-1)\rho$ is a weight of $\op{rad} G \times \tilde{G}$ for all $e\geq 1$.

Let $P=P_I$ be a parabolic subgroup. For $\la \in X(T)_+$, put $\V(\la):= \V(\nabla_L(\la))$ (here $U_I$ acts trivially on $\nabla_L(\la)$). The quotient map $\pi: G/B \to G/P$ induces a quasi-isomorphism
\begin{equation}\label{eq:pushfiber}
\mathbf{R}\pi_* \mc{L}(\la) \cong \V(\la).
\end{equation}

By abuse of notation, we use the same notation for the respective bundles on Schubert varieties that are obtained by restriction. We record the following result.

\begin{lemma}\label{lem:surj}
Let $\la \in X(T)_+$, and $w \in \mc{W}^I$. For all $i\geq 0$ we have $H^i(X(w)_P, \V(\la))  \cong H^i(X(w\cdot w_I), \, \mc{L}(\la))$, and the map induced by restriction is surjective:
\[H^i(G/P, \, \V(\la)) \to H^i(X(w)_P, \V(\la)).\]
Moreover, $H^i(G/P, \, \V(\la))=0$ for $i>0$.
\end{lemma}

\begin{proof}
We have $\pi^{-1} (X(w)_P) = X(ww_I)$ \cite[Section 13.8]{jantzen} and a Cartesian square
\[
\xymatrix@R-0.8pc{ 
G/B \ar[r]^{\pi} & G/P \\
X(ww_I) \ar[u] \ar[r]^{\,\,\pi^w} & X(w) \ar[u]
}\] 
where the vertical maps are inclusions. As $\pi$ is proper and flat, by a base change argument (see \cite[Corollary 12.9]{hartshorne}) and (\ref{eq:pushfiber}) we get $\mathbf{R}\pi_*^w \, \mc{L}(\la) \cong \V(\la)$. This shows that $H^i(X(w)_P, \V(\la))  \cong H^i(X(w w_I), \mc{L}(\la))$ for all $i\geq 0$. The rest of the claims now follows from the diagram above using \cite[Theorem 2]{projnormschub}.
\end{proof}

\subsection{Classes of singularities}\label{sec:fsing}

When $\cha \kk = p > 0$, for a $\kk$-space $V$ and $e\in \bb{Z}_{\geq 0}$ we denote by $V^{(e)}$ the abelian group $V$ with the new $\kk$-space structure $c \cdot v := c^{1/p^e} \cdot v$. When $V$ is a module over an algebraic group $G$, then $V^{(e)}$ also has a $G$-module structure \cite[Section I.9.10]{jantzen}. If $A$ is a $\kk$-algebra, then so is $A^{(e)}$ by using the same multiplicative structure. 

We call a domain $A$ \textit{strongly} $F$\textit{-regular} if for every non-zero $c\in A$ there exists $e>0$ such that the $A^{(e)}$-map $cF^{e}: A^{(e)} \to A$ given by $x \mapsto c x^{p^e}$ is $A^{(e)}$-split.

As we do not need it for our purposes, we refer the reader to \cite{hh} for the definition of $F$-\textit{rational} rings (see (\ref{eq:f-sing}) below for some of its important properties).

When $\cha \kk =0$, an algebraic variety $X$ has \emph{rational singularities}, if for some (hence, any) resolution of singularities $f: Z\to X$ (i.e.\ $Z$ is smooth, and $f$ proper and birational), the natural map $\mc{O}_X \to \mathbf{R}f_* \mc{O}_Z$ is a (quasi-)isomorphism. Further, we say a ring $A$ is of \textit{strongly} $F$\textit{-regular type} if there exist some subring $R$ of $\kk$ which is of finite type over $\bb{Z}$, and some $R$-algebra $A_R$ which is flat of finite type over $R$, such that $A_R \oo_R \kk \cong A$ and for the closed points $\mathfrak{m}$ in a dense open subset of $\op{Spec} R$, the ring $A_R \oo_R R/\mathfrak{m}$ is strongly $F$-regular. 

An affine variety $X$ is $F$-rational (resp.\ strongly $F$-regular or of strongly $F$-regular type) if $\kk[X]$ is so. We have the following implications (where CM stands for Cohen--Macaulay):
\begin{equation}\label{eq:f-sing}
\begin{aligned}
\cha \kk = 0: & \mbox{ regular}  \Rightarrow \mbox{strongly }F\mbox{-regular type} \! \Rightarrow \mbox{rational sing.} \Rightarrow \mbox{normal, CM}; \\
\cha \kk > 0: & \mbox{ regular}  \Rightarrow \, \mbox{strongly }F\mbox{-regular} \,\, \Longrightarrow \, F\mbox{-rational} \Rightarrow \mbox{normal, CM.}
\end{aligned}
\end{equation}
Furthermore, $F$-rationality implies pseudo-rationality \cite{smith} and rational singularities in positive characteristic as defined in \cite{kovacs2}. When $\cha \kk = 0$, a ring has log terminal singularities if and only if it
is of strongly $F$-regular type and $\bb{Q}$-Gorenstein (see \cite{hawa}).

Now let $A$ be a $G$-algebra and $\cha \kk = p >0$. We can assume that $(p^e-1)\rho$ is a weight of $G$ for $e\geq 1$ (otherwise replace $G$ by $\op{rad} G \times \tilde{G}$). Following \cite[Section 4]{hashiunip}, we say that $A$ is $G$-$F$-\textit{pure} if there exists some $e \geq 1$ such that the map $\op{id} \oo F^e : \St_e \oo A^{(e)} \to \St_e \oo A$ splits as a $(G, A^{(e)})$-linear map.

Now we study the coordinate ring of $\ol{BwB} \subset G$, where $w\in \mc{W}$. Consider the rational $B\times T$-subalgebra of $\kk[\ol{BwB}]^{U}$ consisting of dominant $T$-weight spaces
\[\kk[\ol{BwB}]^{U}_+ := \bigoplus_{\la \in X(T)_+} \kk[\ol{BwB}]^{U}_{\la}.\]

Consider the section ring $C(X(w)):= \bigoplus_{\la \in X(T)_+} \! H^0(X(w), \mc{L}(\la))$.

\begin{lemma}\label{lem:cone}
For any $w\in \mc{W}$, we have an isomorphism of $B\times T$-algebras
\[\kk[\ol{BwB}]^{U}_+ \cong C(X(w)).\]
As a consequence, the algebra $\kk[\ol{BwB}]^{U}_+$ is finitely generated, and strongly $F$-regular when $\cha \kk > 0$ (resp.\ of strongly $F$-regular type when $\cha \kk =0$).
\end{lemma}

\begin{proof}
Let $\Gamma := - X(T)_+$ and consider the semigroup ring $\kk[\Gamma]$, which is naturally a subalgebra of $\kk[T]$. We have an isomorphism of $B\times T$-algebras
\[(\kk[\ol{BwB}]^U \oo \kk[\Gamma])^T \cong C(X(w)).\]
On the other hand, we have 
\[(\kk[\ol{BwB}]^U \oo \kk[\Gamma])^T=(\kk[\ol{BwB}]^U_+ \oo \kk[\Gamma])^T =(\kk[\ol{BwB}]^U_+ \oo \kk[T])^T \cong \kk[\ol{BwB}]^U_+,\]
where the second equality follows from the decomposition $\kk[T] \cong \bigoplus_{\la \in X(T)} \kk_{\la}$ as $T$-modules, and the last isomorphism from Lemma \ref{lem:transfer}.

We now show that $C(X(w))$ is finitely generated. By \cite[Theorem 16.2]{grossbook}, $\kk[G]^U$ is finitely generated, and therefore so is $C(X(w_0))=(\kk[G]^U\oo \kk[\Gamma])^T$. By Lemma \ref{lem:surj}, we see that the map $C(X(w_0)) \to C(X(w))$ induced by restriction is onto, hence $C(X(w))$ is finitely generated (alternatively, this follows also from \cite[Theorem 2]{projnormschub}).

Let $\cha \kk > 0$. The (not necessarily noetherian) algebra $\bigoplus_{\la \in X(T)} H^0(X(w), \mc{L}(\la))$ is quasi-$F$-regular, by \cite[Theorem 2.6 (4)]{hashisurjgraded} and the global $F$-regularity of Schubert varieties in the sense of \cite{smith2}, see \cite{schubfreg}, \cite{hashischub}. Therefore, the algebra $C(X(w))$ is also quasi-$F$-regular by \cite[Lemma 2.4]{hashisurjgraded}. The latter is finitely generated, so strongly $F$-regular (see \cite[Section 2.1]{hashisurjgraded}).

For ring $R$, consider the $R$-algebra $C(X(w)_{R})= \bigoplus_{\la \in X(T)_+} \! H^0(X(w)_{R}, \mc{L}(\la)_{R})$. We have $C(X(w)_{\kk'}) = C(X(w)_{\bb{Z}}) \otimes_\bb{Z} k'$ (see \cite[Section II.14.15]{jantzen}), for any field $\kk'$, and $C(X(w)_{\bb{Z}})$ is flat and finitely generated over $\bb{Z}$ (e.g.\ from \cite[Sections II.14.1 and II.14.21]{jantzen}). By \cite[Theorem 5.5]{hh}, $C(X(w)_{\kk'})$ is strongly $F$-regular for a perfect field $\kk'\subset \kk$. This shows that when $\cha \kk = 0$, $C(X(w))$ is of strongly $F$-regular type.
\end{proof}

For the remainder of the subsection, we assume that $\cha \kk >0$.

\begin{lemma}\label{lem:cartanalg}
Let $\Gamma \subset X(T)_+$ be a finitely generated semigroup, and $A= \bigoplus_{\la \in \Gamma} A_\la$ a $\Gamma$-graded integral domain with a $G$-action such that $A_\la\cong \nabla_G(\la)$. Then $A$ is $G$-$F$-pure.
\end{lemma}

\begin{proof}
The proof follows closely that of \cite[Lemma 3]{hashiunipshort}. We can assume that $G=\tilde{G} \times \op{rad} G$. Further, we can assume that the product $\nabla_G(\la) \oo \nabla_G(\mu) \to \nabla_G(\la+\mu)$ in $A$ is given by multiplication of sections of the corresponding line bundles on $G/B$, as seen in the proof of  \cite[Lemma 5.6]{hashisurjgraded}. We denote by $\phi$ the composition of $G$-maps
\[\phi \colon \St \oo A \twoheadrightarrow \bigoplus_{\la \in \Gamma} \St \oo \nabla_G(p\la)  \twoheadrightarrow \bigoplus_{\la \in \Gamma} \nabla_G(p(\la+\rho)-\rho) \xrightarrow{\cong} \St \oo A^{(1)},\]
where the first map is given by projection, the second by multiplication (see \cite[Theorem 1]{projnormschub}), and the third by the inverse of $G$-isomorphism $\St \oo \nabla_G(\la)^{(1)} \xrightarrow{\cong} \nabla_G(p(\la+\rho) - \rho)$ induced also by multiplication of sections (see \cite[Theorem 2.5]{andersen}). Then $\phi$ gives the required splitting, since it is $A^{(1)}$-linear. The latter can be checked on the graded components, where it follows from the commutative diagram (with the obvious maps induced by multiplication):
\[\xymatrix@R-0.3pc@C-0.3pc{
\St \oo \nabla_G(p\la) \oo \nabla_G(\mu)^{(1)} \ar[r] \ar[d] & \nabla_G(p(\la+\rho)-\rho) \oo \nabla_G(\mu)^{(1)} \ar[r]^{\cong} \ar[d] & \St \oo \nabla_G(\la)^{(1)} \oo \nabla_G(\mu)^{(1)} \ar[d] \\
\St \oo \nabla_G(p(\la+\mu)) \ar[r] \quad & \quad \nabla_G(p(\la+\mu+\rho)-\rho)  \ar[r]^{\cong} \quad & \quad \St \oo \nabla_G(\la+\mu)^{(1)} 
}\]
\end{proof}

When $\Gamma$ is saturated, the algebra $A$ as above is strongly $F$-regular \cite[Lemma 5.6]{hashisurjgraded}.

\begin{corollary}\label{cor:lpure}
The algebra $\kk[G]^{U_I \times U}$ is strongly $F$-regular and $L$-$F$-pure.
\end{corollary}

\begin{proof}
The algebra $A=\kk[G]^{U_I \times U}$ has an $L\times T$-action so that we have a decomposition $A=\bigoplus_{\la \in \Gamma} \nabla_L(w_I w_0\la)$ as $L$-modules (e.g.\ see \cite[Theorem 3]{donkinunip}). Clearly, the set $\{w_I w_0 \la \}_{\la \in X(T)_+}$ forms a saturated subsemigroup in the semigroup of dominant weights of $L$. Hence, the claims follow by \cite[Lemma 5.6]{hashisurjgraded} and Lemma \ref{lem:cartanalg}, respectively.
\end{proof}

\subsection{Good filtrations}\label{sec:goodfil}

Take a (possibly infinite-dimensional) $G$-module $V$. Following Donkin \cite{donkin}, an ascending exhaustive filtration 
\[0 = V_0 \subset V_1 \subset V_2 \subset \dots\]
of $G$-submodules of $V$ is a \textit{good filtration} (resp.\ Weyl filtration) of $V$, if each $V_i/V_{i-1}$ is isomorphic to a dual Weyl module (resp.\ to a Weyl module). If $V$ has both good and Weyl filtrations, then we call $V$ \textit{tilting}.

Now let $w\in \mc{W}$. We say that a $B$-module $V$ has a $w$-\textit{excellent filtration}, if it has a $B$-module filtration with successive quotients isomorphic to some $H^0(X(w), \mc{L}(\la))$, with $\la \in X(T)_+$. This is a special type of excellent filtration, as defined in \cite[Definition 2.3.6]{vanderkallen}. Note that a good filtration of a $G$-module is a $w_0$-excellent filtration.

A finite-dimensional $G$-module $W$ \textit{good} if $\Sym_d W^*$ has a good filtration for all $d\geq 0$. In particular, in this case $W$ must have a Weyl filtration. Similarly, we call an affine $G$-variety (resp.\ $B$-variety) $X$ good (resp $w$-excellent) if $\kk[X]$ has a good (resp.\ $w$-excellent) filtration. 

If $X\subset Y$ is a closed $G$-stable subvariety, then we say that $(Y,X)$ is a \textit{good pair} whenever $Y$ is good and the defining ideal $I_X \subset \kk[Y]$ has a good filtration (see \cite[Section 1.3]{donkinconj}). In this case $X$ is automatically good.

If $\cha \kk = 0$, then all (pairs of) affine $G$-varieties are good. An important feature of good filtrations is the following result of Donkin \cite{donkin} and Mathieu \cite[Theorem 1]{mathieu}.

\begin{proposition}\label{prop:mathieu}
If $M$ and $N$ are $G$-modules with good filtrations, then $M\otimes_\kk N$ has a good filtration. In particular, if $X$ and $Y$ are good affine $G$-varieties, then so is $X\times Y$.
\end{proposition}

We list some cases that imply the existence of good filtrations (see \cite[Section 4]{andjan}).

\begin{lemma}\label{lem:good}
Let $V,W$ be finite-dimensional $G$-modules.
\begin{enumerate}
\item If $\langle \chi+ \rho , \al^\vee \rangle \leq \cha \kk$ for all weights $\chi$ of $V$ and all $\al \in \Phi_+$, then $V$ has a good filtration.
\item If $V$ has a good filtration and $\cha \kk >i$, then $\bigwedge^i V$ and $\Sym_i V$ have good filtrations.
\item If $\bigwedge V$ and $ \bigwedge W$ have good filtrations, then $V\otimes W$ is good.
\item $\bigwedge V$ has a good filtration if and only if so does $\bigwedge V^*$ (i.e.\ $\bigwedge V$ is tilting).
\end{enumerate}
\end{lemma}

We further need some basic results.

\begin{lemma}\label{lem:gr}
Let $f: M\to N$ be a $G$-module map. If $M$ has a good filtration and the induced map $M^{U} \to N^{U}$ is onto, then $N$ and $\ker f$ have good filtrations and $f$ is onto.
\end{lemma}

\begin{proof}
Put $I=\im f$ and $K = \ker f$. Fix any $\la \in X(T)_+$. Since $M$ has a good filtration, we have an exact sequence (see \cite[Proposition II.4.16]{jantzen})
\[0 \to \Hom_G(\Delta_G(\la), K) \to \Hom_G(\Delta_G(\la), M) \to \Hom_G(\Delta_G(\la), I) \to \Ext^1_G(\Delta_G(\la), K) \to 0.\]
The assumption gives an exact sequence
\[0 \to K^{U} \to M^{U} \to I^{U} \to 0.\]
Taking $\la$-weights above we obtain that $\Ext^1_G(\Delta_G(\la), K)=0$ (see \cite[Lemma II.2.13]{jantzen}). Since $\la \in X(T)_+$ was arbitrary, this shows that $K$ has a good filtration (see \cite[Proposition II.4.16]{jantzen}), and hence so does $I$. Let $C=\coker f$ and consider an exact sequence $0 \to I \to N \to C \to 0$. Since $I$ has a good filtration, we see as above that the induced sequence $0 \to I^{U} \to N^{U} \to C^{U} \to 0$ is also exact. By assumption $C^{U} = 0$, hence $C=0$.
\end{proof}

\begin{corollary}\label{cor:goodpair}
Let $Y$ be a good affine $G$-variety and $X\subset Y$ a closed $G$-stable subvariety. Then $(Y,X)$ is a good pair if and only if the map $\kk[Y]^{U} \to \kk[X]^{U}$ is surjective.
\end{corollary}

\begin{proof}
If $\kk[Y]^{U} \to \kk[X]^{U}$ is surjective, then it follows from Lemma \ref{lem:gr} that $(Y,X)$ is a good pair. The converse follows from \cite[Proposition 1.4 and Proposition 2]{donkinunip}.
\end{proof}

We introduce a notion for generators of ideals, that is again relevant only in positive characteristic.

\begin{defn}\label{def:goodgen}
Let $Y$ be a good affine $G$-variety and $X\subset Y$ a closed $G$-stable subvariety with defining ideal $I_X \subset \kk[Y]$. We say that a finite set of equations $\mc{P} \subset I_X$ are good defining equations (resp.\ good generators) of $X$ (resp.\ of $I_X$) if the following hold for $M_{\mc{P}} := \op{span}_{\kk} \mc{P}  \, \subset I_X$:
\begin{enumerate}
\item $M_{\mc{P}}$ is a $G$-module with a good filtration;
\item The multiplication map $m_{\mc{P}}: k[Y] \oo M_{\mc{P}} \to I_X$ induces a surjective map on $U$-invariants $(k[Y] \oo M_{\mc{P}})^{U} \to I_X^{U}$.
\end{enumerate}
\end{defn}

Let us record some useful results regarding this notion. We continue with the notation in Definition \ref{def:goodgen}. 

\begin{lemma}\label{lem:exist}
There exist good defining equations for $X \subset Y$ if and only if $(Y,X)$ is a good pair.
\end{lemma}

\begin{proof}
Assume that $(Y,X)$ is a good pair. By \cite[Theorem 16.2]{grossbook}, $\kk[Y]^{U}$ is noetherian, hence $I_X^{U}$ is finitely generated. Choose a finite set of generators. Taking a good filtration of $I_X$, there exists a finite dimensional piece $M$ that contains these generators. We can pick $\mc{P}$ to be a basis of $M$.

Conversely, let $\mc{P} \subset I_X$ be a set of good generators. By Proposition \ref{prop:mathieu}, the domain of the multiplication map $m_{\mc{P}}$ has a good filtration. By Lemma \ref{lem:gr}, we obtain that $m_{\mc{P}}$ is surjective, and $I_X$ has a good filtration.
\end{proof}

The proof above shows assumption (2) in Definition \ref{def:goodgen} can be replaced with the equivalent assumption that $\mc{P}$ generates $I_X$ and $\ker m_{\mc{P}}$ has a good filtration. In particular, the notion does not depend on the choice of the Borel subgroup (see \cite[Remark II.4.16 (2)]{jantzen}). We record another convenient fact.

\begin{lemma}\label{lem:compint}
Assume that $Y$ is good and let $M \subset I_X $ be $G$-module such that a basis $\mc{P}$ of $M$ generates $I_X$ and forms a regular sequence in $\kk[Y]$. Assume that $\bigwedge M$ has a good filtration. Then $\mc{P}$ are good defining equations of $X \subset Y$.
\end{lemma}

\begin{proof}
This follows readily by considering the Koszul resolution, and using \cite[Proposition 3.2.4]{donkin} together with Proposition \ref{prop:mathieu} repeatedly.
\end{proof}

Although we do not need it in this article, the assumption on $\bigwedge M$ in the lemma above can be weakened by requiring only that the good filtration dimension of $\bigwedge^i M$ is at most  $i-1$, for all $i\geq 1$ (see \cite[Section 1.3]{donkinconj}).

\subsection{Deformation of algebras}\label{sec:gropop}

We recall a filtration of algebras considered in \cite{popov} and \cite{gross}. There exists a homomorphism $h: X(T) \to \bb{Z}$ satisfying the following properties:
\begin{enumerate}
\item $h(\la)$ is a non-negative integer for all $\la\in X(T)_+$;
\item if $\chi',\chi \in X(T)$ with $\chi' > \chi$, then $h(\chi')> h(\chi)$.
\end{enumerate}

For a commutative $G$-algebra $A$ over $k$, we define the $\bb{Z}_{\geq 0}$-filtration
\[ F^i A := \{ a \in A \, | \, h(\chi) \leq i \mbox{ for all } T\mbox{-weights } \chi \mbox{ of } \op{span}_\kk  G \cdot a \}.\]
Denote by $\gr A$ the associated graded algebra. Then there is an injective map of $G$-algebras
\begin{equation}\label{eq:grA}
\gr A \hookrightarrow (A^{U^-} \oo_\kk \kk[G/U])^T,  
\end{equation}
which is onto if and only if $A$ has a good filtration \cite[Theorem 16]{gross}.

Consider $L$ a linear algebraic group, and $H\subset L$ a closed subgroup. Let $N:=N_L(H)$ be the normalizer of $H$ in $L$. Let $R$ be an $L$-algebra. The group $N$ acts naturally on $R^H$ and on $H$-invariants $\kk[L]^H=\kk[L/H]$ (by right multiplication). The following is a consequence of \cite[Theorem 4]{popov} (see also  \cite[Theorem 9.1]{grossbook}).
\begin{lemma}\label{lem:transfer}
There is an isomorphism of $N$-algebras $R^H \cong (R \oo_\kk \kk[L/H])^L$.
\end{lemma}

\section{Main results}\label{sec:mainresults}

In this section we develop our general results on collapsing of bundles. We work over an algebraically closed field $\kk$ of arbitrary characteristic (see Remark \ref{rem:field}).  In the special case when $\cha \kk = 0$ and the Schubert variety considered is the flag variety itself, the general framework agrees with that of completely reducible bundles as in \cite{Kempf76}.

We fix the notation that is used throughout the section. Consider a parabolic subgroup $P\subset G$. Without loss of generality, we assume that $P$ is standard corresponding to a set of simple roots $I\subset S$. Let $U_I$ be the unipotent radical of $P$. Let $P= L \ltimes U_I$ be the Levi decomposition, with $L:=L_I$ reductive. We denote by $P^-$ the opposite parabolic subgroup, having decomposition $P^-= L \ltimes  U_I^-$. 

Let $W$ be a finite-dimensional $G$-module. We introduce the map of $L$-modules
\begin{equation}\label{eq:hsplit}
\psi \colon W^{U_I} \longrightarrow \left((W^*)^{U^-_I}\right)^*,
\end{equation}
which is the dual of the composition $(W^*)^{U^-_I} \hookrightarrow W^* \twoheadrightarrow (W^{U_I})^*$. 

\smallskip

Throughout we take an $L$-submodule $V\subset W^{U_I}$ such that the map $\left.\psi \right|_{V} : V \to ((W^*)^{U^-_I})^*$ is injective. The following shows that tracking the map $\left.\psi \right|_{V}$ is relevant only when $\cha \kk > 0$.
\begin{lemma}\label{lem:iso} In either of the following cases, $\left.\psi \right|_{V}$ is an isomorphism:
\begin{itemize}
\item[(a)] $W$ is a semi-simple $G$-module and $V=W^{U_I}$.
\item[(b)] $W= \bigoplus_{i=1}^{n} \Delta_{G}(\la_i)$ for some $\la_i \in X(T)_+$, and $V \subset W^{U_I}$ is $V=\bigoplus_{i=1}^{n} \Delta_{L}(\la_i)$.
\end{itemize}
\end{lemma}

\begin{proof}
For part (a), we can assume that $W$ is a simple $G$-module. Both $W^{U_I}$ and $((W^*)^{U^-_I})^*$ are simple $L$-modules \cite[Proposition II.2.11]{jantzen}, and $\psi$ gives a non-trivial map between their respective highest weight vectors. Therefore, $\psi$ is an isomorphism.

For part (b), we can assume that $W = \Delta_G(\la)$ is a Weyl module. The restriction map $\nabla_G(\la) \to \nabla_L(\la)$ induced by $P/B \subset G/B$ is surjective (see Lemma \ref{lem:surj}). Therefore, the $L$-submodule of $W$ generated by its highest weight vector (of weight $\la$) is $V\cong \Delta_L(\la)$. On the other hand, we have $((W^*)^{U^-_I})^* \cong \Delta_L(\la)$ as $L$-modules (see \cite[Section 1.2]{donkinunip}), generated as an $L$-module by the highest weight vector. Since on the weight space of $\la$ the map $\left.\psi \right|_{V}$ is easily seen to be non-zero, it is also surjective, hence an isomorphism.
\end{proof}

\smallskip

Let $X$ be a closed $L$-stable subvariety of $V$. As $U_I$ acts on $V$ trivially, $X$ is $P$-stable closed subvariety of $W$. We have the following proper collapsing map
\begin{equation}\label{eq:collapse}
q \colon G\times_P X \longrightarrow W,
\end{equation}
with $\op{im} q = G\cdot X$ a closed subvariety of $W$. Let $\pi : G\times_P X \to G/P$ be the bundle map. For any closed subset $Y \subset G/P$, the subvariety $q(\pi^{-1}(Y)) \subset W$ is closed. In the case when $Y=X(w)_P$ is a Schubert variety, then $q(\pi^{-1} (Y)) = \ol{BwX}$ is a $B$-stable subvariety in $X$. 

\begin{proposition}\label{prop:invariantalg}
For any $w\in \mc{W}^I$, the restriction map $\kk[\ol{BwX}] \to \kk[wX]$ induces an isomorphism of algebras 
\[\kk[\ol{BwX}]^{U(w)} \, \xrightarrow{\,\,\cong\,\,} \, \kk[wX].\]
Thus, the algebra $\kk[X]$ is a direct summand of $\kk[\ol{BwX}]$ as a $\kk[X]$-module via $\kk[X]\cong \kk[wX]$.
\end{proposition}

\begin{proof}
The inclusions $wX \subset \ol{BwX} \subset W$ give rise to a commutative diagram
\[
\xymatrix@R-0.5pc{ 
\kk[\ol{BwX}]^{U(w)} \ar[r]^{\quad f} & \kk[wX] \\
\kk[W]^{U(w)} \ar[u] \ar[r]^{\,\,g} & \kk[wV] \ar@{->>}[u]
}\] 
To show that $f$ is onto, it is enough to show that $g$ is so. For this, we show that the map $(W^*)^{U(w)} \to w \cdot V^*$ is onto. As $w^{-1} \cdot (W^*)^{U(w)}=(W^*)^{w^{-1} U w \,\cap\, U^-}$ and $w^{-1} U w \, \cap \, U^- \subset U_I^-$ (cf. (\ref{eq:weylrep})), this follows since the $L$-module map $\left.\psi \right|_{V}^*: (W^*)^{U^-_I} \to V^*$ is onto.

The morphism (\ref{eq:collapse}) induces an injective map of algebras
 \[\kk[\ol{BwX}] \hookrightarrow (\kk[\ol{BwP}] \oo \kk[X])^P.\]
The multiplication map (\ref{eq:opencell}) gives an open immersion into $\ol{BwP}$, inducing an injective map $\kk[\ol{BwP}]^{U(w)} \hookrightarrow \kk[wP]$. The previous maps give  
\[\kk[\ol{BwX}]^{U(w)} \hookrightarrow (\kk[\ol{BwP}]^{U(w)} \!\! \oo \kk[X])^P \hookrightarrow (\kk[wP] \oo \kk[X])^P \cong \kk[wX],\]
thus proving the injectivity of $f$.
\end{proof}

\begin{remark}\label{rem:unipara}
Putting $w=w_0 w_I^{-1}$ in Proposition \ref{prop:invariantalg}, and twisting by $w$ we obtain an isomorphism of $L$-algebras $\kk[G\cdot X]^{U_I^{-}} \xrightarrow{\cong} \kk[X]$.
\end{remark}

\subsection{Good saturations}\label{subsec:good} The following is our main tool for inducing the property of being good via saturations.

\begin{theorem}\label{thm:good}
\begin{itemize}
\item[(a)] The $G$-variety $G\cdot X$ is good if and only if the $L$-variety $X$ is good and the induced map $k[W] \to q_* \mc{O}_{G\times_P X}$ is onto. 
\item[(b)] Assume that $(V,X)$ is a good pair of $L$-varieties and $\left.\psi \right|_{V}$ is a split map of $L$-modules. If there is a good closed $G$-subvariety $Z \subset W$ with $G\cdot X \subset Z$, then $(\,Z\,, \,\, G\cdot X)$ is a good pair.
\item[(c)] Let $Y\subset V$ be a closed $L$-stable subvariety such that $(Y,X)$ is a good pair and $G\cdot Y$ is good. Then $(G\cdot Y\, , \, G\cdot X)$ is a good pair.
\end{itemize}
\end{theorem}

\begin{proof}
Assume that $G\cdot X$ is good. By Remark \ref{rem:unipara} and \cite{donkinunip}, we obtain that $X$ is good. From the proof of Proposition \ref{prop:invariantalg} we have $\kk[G \cdot X]^{U^-_I} \xrightarrow{\,\cong \,}  (q_* \mc{O}_{G\times_P X})^{U^-_I}$, which yields isomorphisms
\begin{equation}\label{eq:unip}
\kk[G \cdot X]^{U^-} \xrightarrow{\,\,\cong \,\,}  (q_* \mc{O}_{G\times_P X})^{U^-} \xrightarrow{\,\,\cong \,\,} \kk[X]^{U_L^-}.
\end{equation}
Therefore, the map $\kk[G \cdot X] \to q_* \mc{O}_{G\times_P X}$ is onto by Lemma \ref{lem:gr}.

Now assume that $X$ is good and $\kk[G \cdot X] \to q_* \mc{O}_{G\times_P X}$ is onto (hence, an isomorphism). By \cite[Theorem 3]{donkinunip}, Proposition \ref{prop:mathieu} and \cite[Proposition 1.2e (iii)]{donkinconj} the $G$-module $q_* \mc{O}_{G\times_P X}= (\kk[G/U_I] \oo \kk[X])^L$ has a good filtration, thus $G\cdot X$ is good.

For part (b), by Corollary \ref{cor:goodpair} the claim is equivalent to the map $\kk[Z]^{U^-} \to \kk[G\cdot X]^{U^-}$ being onto. By Proposition \ref{prop:invariantalg}, it is enough to show that the map $\kk[W]^{U^-} \to \kk[X]^{U_L^-}$ is onto. By Corollary \ref{cor:goodpair}, the map $\kk[V]^{U_L^-} \to \kk[X]^{U_L^-}$ is onto. Hence, the claim follows if we show that the map $\kk[W]^{U^-} \to \kk[V]^{U_L^-}$ is onto. For this, we prove that the restriction of the latter map to the subalgebra $(\Sym ((W^*)^{U_I^-}))^{U_L^-}$ is already onto. 

Since the $L$-map $\left.\psi \right|_{V}$ is split, then so is $\Sym ((W^*)^{U_I^-}) \to \Sym(V^*)$. Therefore, taking $U_L^-$-invariants yields a surjective map.

Now we consider part (c). By Corollary \ref{cor:goodpair} it is enough to see that the morphism $\kk[G\cdot Y]^{U^-} \to \kk[G\cdot X]^{U^-}$ is surjective. By Proposition \ref{prop:invariantalg}, this is equivalent to showing that $\kk[Y]^{U^-_L} \to \kk[X]^{U^-_L}$ is onto. This follows again by Corollary \ref{cor:goodpair}.
\end{proof}

\begin{remark}\label{rem:xi}
Assume $V$ is good and put $\eta= \V(V^*)$ and  $\xi = \V(W^*)/\eta$. Then:
\begin{itemize}
\item[(a)] $G\cdot V$ is good if and only if $H^{i}(G/P, \, \bigwedge^i \xi)=0$, for all $i>0$, by Theorem \ref{thm:good} (a), \cite[Theorem 5.1.2]{weymanbook} and Remark \ref{rem:normalize} below. 
\item[(b)] Assume further that $W$ has Weyl filtration and the $L$-map $\left.\psi \right|_{V}$ is a split. Then using Lemma \ref{lem:gr} we see as in the proof above that the induced map $W^* \to H^0(G/P, \, \eta)$ is onto. Hence, by Theorem \ref{thm:good} (a), $G\cdot V$ is good if and only if the algebra $q_* \mc{O}_{G\times_P V} \cong H^0(G/P, \, \Sym \eta)$ is generated by $H^0(G/P, \, \eta)$.
\end{itemize}
\end{remark}

\begin{corollary}\label{cor:largep}
If $\cha \kk > \dim W$ and $ \cha \kk\geq \langle \chi+ \rho , \al^\vee \rangle$ for all weights $\chi$ of $W$ and all $\al \in \Phi_+$, then $(\,W\,, \, G\cdot V\,)$ is a good pair.
\end{corollary}

\begin{proof}
By Lemma \ref{lem:good} parts (1)--(3), we see that both $V$ and $W$ are good. By \cite[Section 5.6]{jantzen}, both $V$ and $W$ are semi-simple, therefore $\left.\psi \right|_{V}$ is split injective (see Lemma \ref{lem:iso}). The conclusion now follows from Theorem \ref{thm:good} (b).
\end{proof}

If $W$ is as in (\ref{eq:mainsetup}), then putting $X=V$ and $Z=W$ in Theorem \ref{thm:good} (b), we see that $(W, \,G\cdot V)$ is a good pair whenever $\cha \kk > \max \{ \dim \Delta_{G}(\la_i) \,| \, 1\leq i \leq n\}$ by Proposition \ref{prop:mathieu} and Lemma \ref{lem:good}. In particular, $G\cdot V$ is then good as claimed in the Introduction.

\subsection{Singularities via Schubert collapsing}\label{subsec:sing}

Now we turn to Theorems \ref{thm:introsat} and \ref{thm:intromain}. The following result describes the behavior of singularities under collapsing, and it strengthens \cite[Proposition 1 and Theorem 3]{Kempf76} when $w=w_0 w_I^{-1}$ (i.e.\ when $\ol{BwX} = G\cdot X$) in the characteristic zero case as well. 

\begin{theorem}\label{thm:main}
Assume that $G\cdot X$ is good. For $w\in \mc{W}^I$, the $B$-variety $\ol{BwX}$ is $ww_I$-excellent. Furthermore, the following statements hold:
\begin{enumerate}
\item The map $\mc{O}_{\ol{BwX}} \, \xrightarrow{\,\cong\,} \, \mathbf{R} q_*\mc{O}_{\ol{BwP}\times_P X}$ is an isomorphism.
\item $\ol{BwX}$ is normal if and only if $X$ is so.
\item If $\cha \kk = 0$, then $\ol{BwX}$ has rational singularities if and only if so does $X$.
\item If $X$ is an $L$-submodule of $V$, then $G\cdot X$ is strongly $F$-regular (resp.\ of strongly $F$-regular type) when $\cha \kk > 0$ (resp.\ when $\cha \kk =0$), and $\ol{BwX}$ is $F$-rational when $\cha \kk> 0$.
\end{enumerate}
\end{theorem}

\begin{proof}
For part (1), observe that by (\ref{eq:unip}) a good filtration of $\kk[X]$ has composition factors $\Delta_L(\la)$ with such that $\la\in X(T)_+$. By Lemma \ref{lem:surj}, we obtain by induction on filtration that $\mathbf{R}^i q_*\mc{O}_{\ol{BwP}\times_P X} = 0$, for all $i>0$. The map  $\mc{O}_{\ol{BwX}} \to q_*\mc{O}_{\ol{BwP}\times_P X}$ is an isomorphism, since the composition $\kk[W] \to q_*\mc{O}_{G\times_P X} \to q_*\mc{O}_{\ol{BwP}\times_P X}$ is surjective by Theorem \ref{thm:good} (a) and Lemma \ref{lem:surj}.

For part (2), if $\ol{BwX}$ is normal, then by Proposition \ref{prop:invariantalg} so is $X$. Conversely, if $X$ is normal, then so is $\ol{BwX}$ by the normality of $X(w)_P$ \cite{projnormschub} and $\mc{O}_{\ol{BwX}} \cong q_*\mc{O}_{\ol{BwP}\times_P X}$.

Next, we prove the statements regarding $\ol{BwX}$ in part (3) and (4). If $\ol{BwX}$ has rational singularities, then due to the direct summand property in Proposition \ref{prop:invariantalg} so does $X$ according to \cite[Th\'eor\`eme]{boutot}.

Consider the filtration $F^i \kk[X]$ as in Section \ref{sec:gropop}. This gives an exhaustive filtration on $A:=\kk[\ol{BwX}]$ by $F^i A:= (\kk[\ol{BwP}] \oo F^i \kk[X])^P$. The associated graded is
\[\gr A = (\kk[\ol{BwP}]^{U_I} \oo \gr \kk[X])^L \,\stackrel{(\ref{eq:grA})}{\cong}\, (\kk[\ol{BwP}]^{U_I} \oo (\kk[L/U_L]\oo \kk[X]^{U^-_L})^T)^L \cong\]
\[\cong ((\kk[\ol{BwP}]^{U_I} \oo \kk[L/U_L])^L\oo \kk[X]^{U^-_L})^T \cong (\kk[\ol{BwP}]^U\oo \kk[X]^{U^-_L})^T = (\kk[\ol{Bww_IB}]^U_+\oo \kk[X]^{U^-_L})^T, \]
where the last equality is a consequence of $\ol{BwP}=\ol{Bww_I B}$ and (\ref{eq:unip}), and the isomorphism before it follows from Lemma \ref{lem:transfer}. 

Now assume that $X$ has rational singularities when $\cha \kk = 0$ (resp.\ $X$ is an $L$-module when $\cha \kk >0$). By \cite[Theorem 6]{popov} (resp.\ by \cite[Corollary 4.14]{hashiunip}),  $\kk[X]^{U^-_L}$ has rational singularities (resp.\ is strongly $F$-regular). By Lemma \ref{lem:cone} and (\ref{eq:f-sing}), $\kk[\ol{Bww_IB}]^U_+$ has rational singularities (resp.\ is strongly $F$-regular). Hence, $\gr A$ has rational singularities (resp.\ is strongly $F$-regular) by \cite{boutot} (resp.\ \cite[Theorem 5.5]{hh}). As in \cite[Section 5]{popov}, the algebra $\gr A$ is a flat deformation of $A$. Therefore, $A$ has rational singularities by \cite{elkik} (resp.\ is $F$-rational by (\ref{eq:f-sing}) and \cite[Theorem 4.2]{hh}).

Now we show that $G\cdot X$ is strongly $F$-regular in part (4). Let $G' = \tilde{G} \times Z$, with $\tilde{G}$ a covering of $[G,G]$ and $Z \subset T$ a torus so that $G$ is a quotient of $G'$. We can view $W$ as a $G'$-representation. Since $T \subset L$, we have $G\cdot X = \tilde{G} \cdot X$. Moreover, we can lift $P$ to a parabolic $P'$ of $\tilde{G}$ with unipotent radical $U'_I$ and Levi subgroup $L'$. We have $W^{U'_I} = W^{U_I}$ and $(W^*)^{U'^-_I} = (W^*)^{U^-_I}$. Furthermore, $G\cdot X$ (resp.\ $X$) is $G$-good (resp.\ $L$-good) if and only if it is $\tilde{G}$-good (resp.\ $L'$-good) \cite[Section 3]{donkin}. This shows that we can assume that $G$ is simply connected and semisimple. 

Assume that $\cha \kk > 0$. Since $X$ and $G$ are good, using \cite[Theorem 3]{donkinunip} and Proposition \ref{prop:mathieu} we have
\[q_* \mc{O}_{G\times_P X}= (\kk[G/U_I] \oo \kk[X])^L = \left((\kk[G/U_I] \oo \kk[X])^{U_L}\right)^T.\]
As $T$ is linearly reductive, by \cite[Theorem 5.5]{hh} the claim follows once we show that $R:=(\kk[G/U_I] \oo \kk[X])^{U_L}$ is strongly $F$-regular. Since $\kk[X]$ and $\kk[G]$ are factorial rings (see \cite{popovufd}), so is $R$ and $\kk[G]^{U \times U_I}$ (see \cite[Theorem 3.17]{popvin}). In particular, since $\kk[G]^{U \times U_I}$ is Cohen--Macaulay by Corollary \ref{cor:lpure} and (\ref{eq:f-sing}), it is Gorenstein \cite{murthy}.

We have an action of $G$ on $R$ induced from its left action on $\kk[G]$. We have an isomorphism $R \cong (\kk[L/U_L] \oo \kk[G/U_I] \oo \kk[X])^L$, which is easily seen to be $G$-equivariant. The algebra $\kk[L/U_L] \oo \kk[G/U_I] \oo \kk[X]$ has a good filtration as a $G\times L$-module, as seen using \cite[Theorem 3]{donkinunip} and Proposition \ref{prop:mathieu}. By  \cite[Proposition 1.2e (iii)]{donkinconj}, we obtain that $R$ has a good filtration as a $G$-module. We consider the invariant ring $R^{U}$. By Corollary \ref{cor:lpure}, \cite[Theorem 5.2]{hashisurjgraded} and \cite[Theorem 4.4 and Lemma 4.7]{hashiunip}, the $\bb{Z}_{\geq 0}$-graded ring $\kk[G]^{U \times U_I}  \oo \kk[X]$ is Gorenstein, strongly $F$-regular, and $L$-$F$-pure. Then \cite[Corollary 4.13]{hashiunip} implies that $R^{U}$ is strongly $F$-regular. Using the filtration in Section \ref{sec:gropop}, this implies that $R$ is $F$-rational by (\ref{eq:f-sing}) and \cite[Theorem 4.2]{hh} (see also \cite[Corollary 3.9]{hashiunip}). Since $R$ is factorial and Cohen--Macaulay, it is also Gorenstein \cite{murthy}. This shows that $R$ is strongly $F$-regular (see \cite[Corollary 4.7]{hh} or \cite{hh2}). 

Now let $\cha \kk = 0$. We can choose a suitable large set of primes $S$ such that for $D=\bb{Z}[S^{-1}]$ we have:  the map $G\times W \to W$ (resp.\ $G\times X \to G\cdot X$) is defined over $D$;
 $G\cdot X = (G_D\cdot X_{D}) \times_{\op{Spec} D} \op{Spec}(\kk)$; the affine scheme $(G\cdot X)_{D}=G_D\cdot X_{D}$ is flat over $D$; both $W_{\ol{\bb{F}}_p}$ and  $(G\cdot X)_{\ol{\bb{F}}_p}$ are good for $p\notin S$ (see Corollary \ref{cor:largep}); $W_{\ol{\bb{F}}_p}$ (resp.\ $X_{\ol{\bb{F}}_p}$) is a semi-simple $G_{\ol{\bb{F}}_p}$-module (resp.\ $L_{\ol{\bb{F}}_p}$-module) (see \cite[Section II.5.6]{jantzen}). For such $p\notin S$, for $V=X_{\ol{\bb{F}}_p}$ the map $\left.\psi \right|_{V}$ in (\ref{eq:hsplit}) is injective (see Lemma \ref{lem:iso}). By the previous paragraph and \cite[Theorem 5.5]{hh}, we obtain that $(G \cdot X)_{\bb{F}_p}$ is strongly $F$-regular. Hence, $G\cdot X$ is of strongly $F$-regular type.
\end{proof}

\begin{remark}\label{rem:field}
As seen in the proof above, the assumption on the field to be algebraically closed is not essential. The claims about rational singularities and strongly $F$-regular type (resp.\ $F$-rational singularities) hold over any field, e.g.\ by \cite{boutot} (resp.\ proof of \cite[Lemma 1.4]{smith}), as do claims (1) and (2). The claim on strong $F$-regularity holds for any $F$-finite (e.g.\ perfect) field \cite[Theorem 5.5]{hh}.
\end{remark}

\begin{remark}\label{rem:normalize}
Even if $X$ is good, it may happen that $G\cdot X$ is not, as can be seen in Example \ref{ex:counter}. Nevertheless, we still have  $\mathbf{R}^i q_*\mc{O}_{\ol{BwP}\times_P X}=0$ for $i>0$. Further if $X$ is good, normal, and $q: \ol{BwP} \times_P X \to \ol{BwX}$ is birational (or, more generally, the generic fiber of $q$ is connected and $q$ is separable, as in \cite[Theorem 2.1 (a)]{lorwey}), 
then the results in Theorem \ref{thm:main} carry over if we replace the variety $\ol{BwX}$ in each statement (besides part (2)) with its normalization, which is then in turn a $ww_I$-excellent variety.
\end{remark}

We further note that if one knows a good filtration of $\kk[X]$ explicitly, then by Theorem \ref{thm:main} one obtains readily a corresponding $ww_I$-excellent filtration for $\kk[\ol{BwX}]$. It is then possible to compute the ($T$-equivariant) Hilbert function for $\kk[\ol{BwX}]$ using Lemma \ref{lem:surj} and the Demazure character formula (e.g.\ \cite[Corollary 3.3.11]{brionkumar}). 

By Proposition \ref{prop:invariantalg} and \cite[Theorem 5.5]{hh} if $\ol{BwX}$ is strongly $F$-regular (when $\cha \kk > 0$), then $X$ must also be strongly $F$-regular. In the case of a Borel subgroup, we can strengthen Theorem \ref{thm:main} by giving the following converse to this statement.

\begin{corollary}\label{cor:borel}
Assume that $P=B$ is a Borel subgroup and $W$ has a Weyl filtration. Then $G\cdot X$ is good. Moreover, for $w\in \mc{W}$, the variety $\ol{BwX}$ is strongly $F$-regular (resp.\ of strongly $F$-regular type) when $\cha \kk > 0$ (resp.\ when $\cha \kk =0$) if and only if so is $X$.
\end{corollary}

\begin{proof}
We can assume that $P=B$. Since $T$ is linearly reductive, $(V,X)$ is a good pair. By Theorem \ref{thm:good} (c), in order to show that $G\cdot X$ is good it is enough to show that $G\cdot V$ is so. For this, we use Theorem \ref{thm:good}(a). Since $V\subset W^U$, we have a $T$-decomposition $V=\bigoplus_{i=1}^n \kk_{\la_i}$, where $\la_i \in X(T)_+$. The section ring 
\[q_* \mc{O}_{G\times_B V} = \bigoplus_{(m_i)\in \bb{N}^n} H^0(\mc{L}(\sum_{i=1}^n m_i\la_i))\]
is generated in the components of the unit tuples, i.e.\ by the sum $\bigoplus_{i=1}^{n} \nabla_{G}(\la_i)$, as it follows from \cite{projnormschub} (see also \cite{multicone}). By Remark \ref{rem:xi} (b), $G\cdot V$ is good.

Assume that $X$ is strongly $F$-regular. Note that both $\kk[\ol{BwB}]^U_+$ and $\kk[X]$ are $X(T)_+$-graded algebras, so also $\bb{Z}_{\geq 0}$-graded, using for instance the map $h$ in Section \ref{sec:gropop}. Then the algebra $q_* \mc{O}_{\ol{BwB}\times_B X} = (\kk[\ol{BwB}]^U \oo \kk[X])^T=(\kk[\ol{BwB}]^U_+ \oo \kk[X])^T$ is strongly $F$-regular, as it follows by combining Lemma \ref{lem:cone},  \cite[Theorem 5.2]{hashisurjgraded} and \cite[Theorem 5.5]{hh}. Since $G\cdot X$ is good, the conclusion follows from Theorem \ref{thm:main} (1).

Now let $\cha \kk = 0$. Assume $X$ is of strongly $F$-regular type, and consider a finitely generated $\bb{Z}$-algebra $R \subset k$ as in the definition in Section \ref{sec:fsing} (enlarging, if necessary, so that the action of $T_R$ is well-defined). Let $(\ol{BwX})_R = \op{Spec}((C(X(w)_R) \oo R[X_R])^{T_R})$. As in the proof of Lemma \ref{lem:cone}, $(\ol{BwX})_R$ is flat of finite type over $R$, and $(C(X(w)_R) \oo R[X_R])^{T_R} \oo_R \kk' \cong (C(X(w)_{\kk'}) \oo \kk'[X_{\kk'}])^{T_{\kk'}}$, for any field $\kk'$ over $R$ (see \cite[Section I.2.11]{jantzen}). By Theorem \ref{thm:main} (1), we have $(\ol{BwX})_R \times_{\op{Spec}(R)} \op{Spec}(\kk) \cong \ol{BwX}$. When $\kk'$ is a residue field of $R$, it is finite, in which case $C(X(w)_{\kk'})$ is strongly $F$-regular, as seen in the proof of Lemma \ref{lem:cone}. As in the previous paragraph, we conclude that $(\ol{BwX})_{R/\mathfrak{m}}$ is strongly $F$-regular for maximal ideals $\mathfrak{m}$ in a dense open subset of $\op{Spec}(R)$.

Finally, if $\ol{BwX}$ is of strongly $F$-regular type, using Proposition \ref{prop:invariantalg} we see by an argument similar to the above that $X$ is also of strongly $F$-regular type.
\end{proof}

Further, we provide a result that can lead to more general varieties outside the equivariant setting. Following \cite{brionmultfree}, we call a closed subvariety $Y\subset G/P$ multiplicity-free if it is rationally equivalent to a multiplicity-free linear combination of Schubert cycles.

\begin{corollary}\label{cor:multfree}
Let $Y$ be a multiplicity-free subvariety of $G/P$, and assume that $G\cdot X$ is good. Then $\mc{O}_{q(\pi^{-1}(Y))} \, \xrightarrow{\,\cong\,} \, \mathbf{R} q_*\mc{O}_{\pi^{-1}(Y)}$ is an isomorphism. Moreover, if $X$ is normal (resp.\ has rational singularities when $\cha \kk  = 0$), then $q(\pi^{-1}(Y))$ is normal (resp.\ has rational singularities).
\end{corollary}

\begin{proof}
The proof of the isomorphism $\mc{O}_{q(\pi^{-1}(Y))} \, \xrightarrow{\,\cong\,} \, \mathbf{R} q_*\mc{O}_{\pi^{-1}(Y)}$ follows as in Theorem \ref{thm:main} (a) using \cite[Theorem 0.1]{brionmultfree} and Lemma \ref{lem:surj}. The claim on normality follows from this, as $Y$ itself is normal \cite[Theorem 0.1]{brionmultfree}. Moreover, $Y$ has rational singularities when $\cha \kk = 0$  \cite[Theorem 0.1 and Remark 3.3]{brionmultfree}, hence we conclude that so does $q(\pi^{-1}(Y))$ by \cite[Theorem 1]{kovacs}.
\end{proof}

\subsection{Defining equations of saturations}\label{subsec:def} In this section we give a result on the defining equations of $G\cdot X$ in $W$. Assume that $G\cdot V$ is good. Let $M \subset \kk[V]$ be an $L$-stable module with a good filtration. We can associate to it a $G$-module $M' \subset \kk[G\cdot V]$ in the following way. Consider the inclusion of sheaves $\mc{V}(M) \subset \mc{V}(\Sym V^*)$ on $G/P$. Then we put $M' = H^0(G/P, \mc{V}(M))$. As in the proof of Theorem \ref{thm:main} (1), we see that $M'$ has a good filtration as a $G$-module. Note that $M'$ contains $\op{span}_{\kk} G\cdot M$ via the inclusion given by Remark \ref{rem:unipara}, and this containment is an equality when $M'$ is a semi-simple $G$-module.

\begin{theorem}\label{thm:defi}
Let $(V,X)$ be a good pair with $G\cdot V$ good, and denote by $I_X \subset \kk[V]$ the defining ideal of $X \subset V$. Let $M$ be the span of a set of good generators of $I_X$ and take a basis $\mc{P}'$ of the $G$-module $M'\subset \kk[G\cdot V]$ associated to $M$ as above. Consider the following:
\begin{enumerate}
\item A set of generators $\mc{P}_{G\cdot V}$ of the defining ideal $I_{G\cdot V} \subset \kk[W]$ of $G\cdot V$;
\item A lift $\tilde{\mc{P}'} \subset \kk[W]$ of the set $\mc{P}' \subset \, \kk[W]/I_{G\cdot V}$.
\end{enumerate}
Then the defining ideal of $G \cdot X$ in $\kk[W]$ is generated by $\mc{P}:= \, \mc{P}_{G\cdot V} \, \cup \, \tilde{\mc{P}'}$. 

Furthermore, assume that $(W, \, G\cdot V)$ is a good pair. If either $M'$ is a tilting module, or there are no dominant weights $\la >\mu$ such that $(M')^{U}_\la \neq 0 \neq (I_{G\cdot V})^{U}_\mu$ , then the lift $\tilde{\mc{P}'}$ can be chosen such that $\op{span}_{\kk} \tilde{\mc{P}'} \, \subset \kk[W]$ is $G$-stable; with such lift, if $\mc{P}_{G\cdot V}$ are good generators of $I_{G\cdot V}$ then $\mc{P}$ is a set of good defining equations of $G\cdot X \, \subset W$.
\end{theorem} 

\begin{proof}
Let $J \subset \kk[G\cdot V]$ denote the defining ideal of $G\cdot X$ in $G\cdot V$. We have an exact sequence
\[0 \to J \to \kk[G\cdot V] \to\kk[G\cdot X] \to 0.\] 
By Remark \ref{rem:unipara}, taking $U^-_I$-invariants in the sequence above we get that $J^{U^-_I} \cong I_X$. Furthermore, by construction $M'\subset J$ and $M\subset M'^{U^-_I}$. Consider the multiplication map 
\[m_{\mc{P}'}: \kk[G\cdot V] \oo M' \to J.\]
By Lemma \ref{lem:gr} and Proposition \ref{prop:mathieu}, to see that $m_{\mc{P}'}$ is surjective, it is enough to show that the induced map on $U^-$-invariants is so. This is a consequence of the fact that the following composition of maps is surjective by the assumption on good generators of $I_X$:
\begin{equation}\label{eq:ontoinv}
(\kk[V] \oo M)^{U_L^-} \hookrightarrow (\kk[G\cdot V]^{U^-_I} \oo M'^{U^-_I})^{U_L^-} \hookrightarrow (\kk[G\cdot V] \oo M')^{U^-} \to J^{U^-} = I_X^{U^-_L}.
\end{equation}
As $\mc{P}'$ generates $J = I_{G\cdot X}/ I_{G\cdot V}$, it is clear that $\mc{P}$ generates $I_{G\cdot X}$.

Let $N$ be the $G$-submodule $N \subset I_{G\cdot X}$ corresponding to $M' \subset J$. We have an exact sequence
\[0 \to I_{G\cdot V} \to N \to M'\to 0.\]
To show that $\tilde{\mc{P}'}$ can be chosen in the required way, we show that the sequence splits as $\Ext^1_G(M', I_{G\cdot V}) = 0$. When $M'$ is tilting, this is a consequence of \cite[Proposition II.4.13]{jantzen}, as $I_{G\cdot V}$ has a good filtration and $M'$ has a Weyl filtration. The other case is a consequence of \cite[Proposition 2]{friedlander}.

By the splitting above, we have $M_{\tilde{\mc{P}'}} = \op{span} \tilde{\mc{P}'} \cong M' $ as $G$-modules. It has a good filtration, as the module $M_{\mc{P}_{G\cdot V}}$, since $\mc{P}_{G\cdot V}$ is a set of good generators. Therefore, $M_{\mc{P}} = M_{\tilde{\mc{P}'}} \oplus M_{\mc{P}_{G\cdot V}}$ has a good filtration \cite[Corollary 3.2.5]{donkin}. Consider the commutative diagram
\[\xymatrix@R-0.4pc@C-0.3pc{ 
0 \ar[r] & (k[W] \oo M_{\mc{P}_{G\cdot V}})^{U} \ar[r]\ar[d] & (k[W] \oo M_{\mc{P}})
^{U} \ar[r]\ar[d] & (k[W] \oo M_{\tilde{\mc{P}'}})^{U} \ar[r]\ar[d] & 0 \\
0 \ar[r] & \quad (I_{G\cdot V})^{U} \quad \ar[r] & \quad (I_{G\cdot X})^{U} \quad \ar[r] & \quad J^{U} \quad \ar[r] & 0
}\]
Due to the respective modules having good filtrations by Proposition \ref{prop:mathieu}, the rows of the diagrams are exact  \cite[Proposition 1.4 and Proposition 2]{donkinunip}. Since $\mc{P}_{G\cdot V}$ is a set of good generators, the first vertical map is onto. We are left to show that the third vertical map is onto, or equivalently, that the following composition is surjective (see comment after Lemma \ref{lem:exist}):
\[ (\kk[W] \oo M_{\tilde{\mc{P}'}})^{U^-} \to (\kk[G\cdot V] \oo M')^{U^-} \to J^{U^-}.\]
The first map is onto since $M_{\tilde{\mc{P}'}} \xrightarrow{\cong} M'$ and $(W, \, G\cdot V)$ is a good pair. The second map is onto as seen in (\ref{eq:ontoinv}). Thus, $\mc{P}$ is a good generating set of $I_{G\cdot X}$.
\end{proof}

\begin{remark}
With the assumptions above, one can similarly give defining equations of $\ol{BwX}$, provided we have defining equations of $\ol{BwV}$ in $\kk[G\cdot V]$.
\end{remark}

When $G\cdot V$ is good, by Theorem \ref{thm:main} one can in principle apply \cite[Theorem 5.1.3]{weymanbook} to obtain a (minimal) set of generators $\mc{P}_{G\cdot V}$ (as seen in Remark \ref{rem:xi}), or even its minimal free resolution. We note that the minimal free resolution of $G\cdot V$ given by \textit{loc.\ cit.}\ has length equal to $\codim_{G\cdot V} W$, since $G\cdot V$ is Cohen--Macaulay (\ref{eq:f-sing}). For variations of this technique, see for example \cite[Section 6]{weymanbook} or \cite[Proposition 4.4]{kinlor}.

\section{Special cases and applications}\label{sec:applications}

This section is devoted to demonstrate the strength of our results through some important applications, both classical and new. The examples in the next three subsections fit into the situation described in the Introduction (\ref{eq:mainsetup}). 

\subsection{Varieties of determinantal type}\label{subsec:det}

Let $m\geq n\geq 0$, and consider the case when $W$ is the space of $m\times n$ matrices, $n\times n$ skew-symmetric matrices, or $n\times n$ symmetric matrices -- the latter can be also identified with the 2nd divided power of $\kk^n$. Then we choose $G$ to be $\GL(m)\times \GL(n)$, $\GL(n)$ or $\GL(n)$, and $W=\Delta_G(\la)$ to be $\kk^m \oo \kk^n$, $\bigwedge^2 \kk^n$, or $\Delta_G(2\omega_1)$, respectively.
For $0\leq r \leq n$, we put $L$ to be $\GL(r)\times \GL(r)$, $\GL(r)$ or $\GL(r)$, respectively (and $V=\Delta_L(\la)$). Then $G\cdot V$ is precisely the closed subvariety in $W$ of matrices of rank at most $r$ (see \cite[Section 6]{weymanbook}).

The variety $W$ (resp.\ $V$) is good in arbitrary characteristic (see Lemma \ref{lem:good} and \cite{boffi}). Thus, by Theorem \ref{thm:introgood} (with $X=V$) the $G$-variety $G\cdot V$ is good as well. Therefore, by Theorem \ref{thm:main} $G\cdot V$ is strongly $F$-regular when $\cha \kk >0$ (resp.\ is of strongly $F$-regular type when $\cha \kk = 0$) and $\ol{BwV}$ is $F$-rational (resp.\ has rational singularities if $\cha \kk = 0 $). This yields all $G$-orbit closures $G\cdot V$ and many $B$-orbit closures $\ol{BwV}$ in $W$. 

For $G$-orbit closures in the case of symmetric matrices, this answers \cite[Question 5.10]{katzmir}. For $G$-orbit closures in  $\kk^m \oo \kk^n$ and $\bigwedge^2 \kk^n$, we recover the results \cite{hh2}, \cite[Theorem 1.3]{baetica} (see also \cite[Chapter 7]{baeticabook}). 

The $B$-orbit closures are called matrix Schubert varieties in the literature. As far as we are aware, in this case the results are new even in characteristic $0$, except in the space of $m\times n$ matrices, when it is known that all matrix Schubert varieties are strongly $F$-regular, as this can be reduced to the corresponding statement on Schubert varieties \cite{schubfreg} (see Corollary \ref{cor:borel}) by an identification as done in \cite{fulton}.

Let us show that the $(r+1)\times (r+1)$ minors of a generic symmetric matrix give good defining equations for the space of symmetric matrices of rank $\leq r$ in $W$ using Theorem \ref{thm:defi} (the other cases are analogous and slightly easier). We work on downwards induction on $r$, the case $r=n$ being trivial. Let $V$ be the space of $r\times r$ symmetric matrices as above, and consider $X\subset V$ the matrices of rank $< r$. Clearly, the symmetric determinant is a good defining equation for $X\subset V$ (e.g.\ Lemma \ref{lem:compint}). The associated $G$-module $M'\subset \kk[G\cdot V]$ in Theorem \ref{thm:defi} is $M'=\nabla_G(2 \omega_r)$, and it is easy to see that it satisfies the condition that there are no dominant weights $\mu < 2 \omega_r$ with $(I_{G\cdot V})^{U}_\mu \neq 0$. The lift $\tilde{\mc{P}}'$ can be chosen to be the $r \times r$ minors of a generic symmetric matrix, while $\mc{P}_{G\cdot V}$ are the $(r+1)\times (r+1)$ minors, by the induction hypothesis. By Theorem \ref{thm:defi}, we conclude that $\tilde{\mc{P}}'$ is a good set of defining equations for $G\cdot X$ in $W$.

\subsection{Varieties of complexes on arbitrary quivers}\label{sec:radsquare}

The geometry of the Buchsbaum--Eisenbud varieties of complexes has been investigated thoroughly in a number of articles. In \cite{Kempfcomp} it has been shown that these varieties have rational singularities in characteristic zero, based on the method in \cite{Kempf76}. A characteristic-free approach has been pursued in \cite{DecoStrick} using Hodge algebras, where defining equations are provided as well. In characteristic zero, this result has been proved also in \cite{brioncomp} by showing that their algebra of covariants is a polynomial ring.
Frobenius splitting methods have been applied in \cite{MT2}. One can realize such varieties as certain open subsets in Schubert varieties \cite{zele}, \cite{LM98}.  Similar varieties have been studied in \cite{Strick1}, \cite{Strick2}, \cite{MT1} for other special quivers. These varieties can be considered for any quiver, and are particular cases of certain rank varieties of radical square zero algebras, as in explained in \cite{kinlor}. In \emph{ibid.}, it is shown that in characteristic zero all such varieties have rational singularities, and defining equations are provided. We explain now how to extend such results to arbitrary characteristic, as announced in Remark 4.16 of  \emph{ibid.} Additionally, we obtain analogous results for $B$-varieties.

We follow closely the notation established in \cite{kinlor}. Consider the (associative, non-commutative) radical square zero algebra $A=\kk Q/\kk Q_{\geq 2}$, with $Q$ an arbitrary finite quiver with the set of vertices $Q_0$ and arrows $Q_1$. For a dimension vector $\bd: Q_0 \to \bb{Z}_{\geq 0}$, we consider the representation space
\[
\rep_Q(\bd) = \prod_{\za \in Q_1} \Hom_\kk(\kk^{\bd(t\za)}, \kk^{\bd(h\za)}) = \bigoplus_{\za \in Q_1} (\kk^{\bd(t\za)})^* \oo \kk^{\bd(h\za)},
\]
and within the representation variety of $A$
\[\rep_A(\bd) = \{ M \in \rep_Q(\bd) \, \mid \, M_{\beta}\circ M_{\za}=0, \mbox{ for all } \za,\beta\in Q_1 \mbox{ with } h\za = t\beta\},\]
which has a natural action of the reductive group $\GL(\bd) = \prod_{x \in Q_0} \GL(\bd(x))$. For $x\in Q_0$ and $M\in \rep_Q(\bd)$, we put 
\[h_x(M)=\bigoplus_{h\za =x} M_{\za} \colon  \, \bigoplus_{h\za =x} M_{t\za} \to M_x.\]

For a dimension vector $\br \leq \bd$, we denote by $C_\br$ the closure of the set of representations $M \in \rep_A(\bd)$ such that $\rank h_x(M)=\br(x)$, for all $x\in Q_0$. Let $\bs=\bd - \br$. By \cite[Theorem 3.19]{kinlor} the variety $C_\br$ is irreducible, and it is non-empty if and only if 
\begin{equation}\label{eq:nonzero}
\sum_{h\za = x} \bs(t\za) \,  \geq \br(x), \mbox{ for all }  x\in Q_0.
\end{equation}
Furthermore, each irreducible component of $\rep_A(\bd)$ is of the form $C_\br$, for some $\br\leq \bd$. 

Now fix $\br \leq \bd$ as in (\ref{eq:nonzero}). With the notation from Section \ref{sec:mainresults}, we let $W=\rep_Q(\bd)$, $V=\bigoplus_{\za \in Q_1} (\kk^{\bs(t\za)})^* \oo \kk^{\br(h\za)}$, $G=\GL(\bd)$, $L= \prod_{x \in Q_0} (\GL(\bs(x)) \times \GL(\br(x)))$ . It is implicit from the proof of \cite[Theorem 3.19]{kinlor} that $C_\br = G\cdot V$ (in fact, the collapsing map $q: G\times_P V \to C_\br$ is a resolution of singularities). The variety $W$ (resp.\ $V$) is good in arbitrary characteristic by Lemma \ref{lem:good} and Proposition \ref{prop:mathieu}. Thus, by Theorem \ref{thm:introgood} the $G$-variety $C_\br$ is good and Theorem \ref{thm:main} implies the following result.
\begin{corollary}\label{cor:nodefreg}
The rank variety $C_\br$ is strongly $F$-regular when $\cha \kk >0$ (resp.\ of strongly $F$-regular type when $\cha \kk =0$).
\end{corollary}

Moreover, the varieties $\ol{BwV} \subset C_\br$ are $F$-rational when $\cha \kk > 0$ (resp.\ have rational singularities when $\cha \kk = 0$). Note that the Buchsbaum--Eisenbud varieties of complexes are spherical (e.g.\ \cite{brioncomp}), therefore such varieties are always $B$-orbit closures in this case as there are only finitely many $B$-orbits \cite{brisph}, \cite{vinsph}. We leave the details of the combinatorial characterization of such $B$-orbit closures to the interested reader. 

In \cite[Corollary 4.13]{kinlor}, explicit defining equations are provided for all $C_\br$ when $\cha \kk = 0$. 
We give a self-contained argument to show that, in the case when $Q$ has no loops, these equations are also defining equations when $\cha \kk > 0$.

For $\za \in Q_1$, we let $X_{\za}$ be the $\bd(t\za)\times \bd(h\za)$ generic matrix of variables. We identify the coordinate ring $\kk[\rep_{\kk Q}(\bd)]$ with a polynomial ring in the entries of the matrices $\{X_\za\}_{\za\in Q_1}$. For $x \in Q_0$, we write $H_x$ (resp.\ $T_x$) for the $\bd(x)\times \left(\displaystyle\sum_{h\za = x} \bd(t\za)\right)$ matrix (resp.\ $\left(\displaystyle\sum_{t\za = x} \bd(h\za)\right)\times \bd(x)$ matrix) obtained by placing the matrices $X_\za$ with $h\za = x$ next to (resp.\ with $t\za = x$ on top of) each other. 

\begin{corollary}\label{cor:nodedefi}
Assume $Q$ has no loops, and let $C_\br \subset \rep_A(\bd)$ be non-empty. The following set of polynomials in $\kk[\rep_{\kk Q}(\bd)]$ form a good set of generators for the prime ideal of $C_\br$, as $x$ runs through all the vertices in $Q_0$:
\begin{enumerate}
\item The $(\br(x)+1)\times (\br(x)+1)$ minors of $H_x$;
\item The $(\bs(x)+1) \times (\bs(x)+1)$ minors of $T_x$;
\item The entries of \, $T_x \cdot H_x$;
\end{enumerate}
\end{corollary}

\begin{proof}
We work by splitting nodes one at a time, analogously to \cite[Corollary 4.13]{kinlor}. We note that in Theorem \ref{thm:defi}, the module $M'$ is tilting in this case (see Lemma \ref{lem:good} (4)). To conclude using Theorem \ref{thm:defi} as in \cite[Corollary 4.13]{kinlor}, we are reduced to show that the equations (1)--(3) with $x=2$ are good defining equations of $C_\br$ for the following quiver (compare with \cite[Proposition 4.4]{kinlor})
\[\xymatrix{ 1 \ar[r]^a & 2 \ar[r]^b & 3 } \]
As in the case of determinantal varieties in Section \ref{subsec:det}, we can further reduce using Theorem \ref{thm:defi} (applied at vertices $1$ and $3$) to the case $\br=(0,d_1,d_3)$ (when we have $d_2 \geq d_1 + d_3$). In such case only the equations of type (3) appear, and they form a regular sequence. Using the Jacobian criterion, one readily obtains that the ideal generated by these polynomials is radical. Moreover, by Lemma \ref{lem:compint} they give good defining equations for $C_\br \subset \rep_{\kk Q}(\bd)$, thus yielding the conclusion.
\end{proof}

The article \cite{kinlor} further demonstrates the usefulness of working in the relative situation $X\subset V$. By splitting nodes one at a time, the method is applied to a large number of other quiver varieties in characteristic zero. The main obstruction to extending such results to positive characteristics readily is that so far the good property of the corresponding $L$-variety $X$ has been studied only in a handful of cases (e.g.\  \cite{donkinconj}).

\subsection{Further examples}

When $G=\GL(n)$, $L=\GL(r)$ (with $r\leq n$), $W=\Delta_G(\la)$ and $V=\Delta_L(\la)$, the variety $G\cdot V$ is called higher rank variety \cite[Section 7]{weymanbook}. Thus, Theorem \ref{thm:main} generalizes Proposition 7.1.2 in \emph{loc.\ cit.}\ to characteristics that are not \lq\lq too small", and further gives new results for the varieties $\ol{BwV}$. We note that the result does not hold in arbitrary characteristic, as the following example shows.

\begin{example}\label{ex:counter} Let $G=\GL(3)$, $W=\bigwedge^3 \kk^6$, $V=\bigwedge^3 \kk^5$ with $\cha \kk = 2$. Then $V$ is a good variety, but $G\cdot V$ is not normal, as shown by Weyman \cite[Proposition 7.3.10]{weymanbook}. Using Theorems \ref{thm:good} and \ref{thm:main} we see that $W$ is not good (nor is the hypersurface given by the discriminant of degree $4$), a fact further observed in \cite[Example 3.3]{kallen}. Nevertheless, by Remark \ref{rem:normalize} the normalization of $G\cdot V$ is strongly $F$-regular.

We can extrapolate this to higher dimensions as follows. Set $X:=G\cdot V$ from above. Let $n\geq 6$, and consider inclusions $\bigwedge^3 \kk^5 \subset \bigwedge^3 \kk^6 \subset \bigwedge^3 \kk^n$. Then the saturation $Y:=\GL(n) \cdot \bigwedge^3 \kk^5 \, \subset \, \bigwedge^3 \kk^n$ is the same as the saturation $\GL(n) \cdot X \, \subset \GL(n)\cdot \bigwedge^3 \kk^6 \, \subset \bigwedge^3 \kk^n$. We have seen that $X$ is not strongly $F$-regular, hence neither is $Y$ by Proposition \ref{prop:invariantalg} and \cite[Theorem 5.5]{hh}, but the normalization of $Y$ is again strongly $F$-regular by Remark \ref{rem:normalize}. In particular, $\bigwedge^3 \kk^n$ is not good by Theorems \ref{thm:good} and \ref{thm:main}.
\end{example}

Other examples of saturations $G\cdot V$ (and $\ol{BwV}$) where our results can be readily applied include varieties considered in \cite[Section 2]{Kempf76}, \cite{multicone}, \cite{samwey}, \cite{marcus}, \cite{landwey2}, and the subspace varieties in  \cite{landwey1} (including the relative setting for secant varieties, as in \cite[Proposition 5.1]{landwey1}), thus strengthening the corresponding results therein.

As explained in the Introduction, the results can be effectively used in the study of the geometry of orbit closures for any representation $W$ (as in (\ref{eq:mainsetup})) of a reductive group. Since such problems have been pursued intensively in numerous articles for various special representations, it would be difficult to list them all in relation with our results. 
We simply direct the reader to \cite{weymanbook} and the references therein for a large collection of such examples.

\subsection{Vanishing results for bundles on Schubert varieties}\label{subsec:ample} 

First, we record the following positive characteristic version of the Grauert--Riemenschneider theorem for collapsing of bundles (cf. \cite[Section 3]{Kempf76}). Such results are of interest (see \cite[Theorem 1.3.14]{brionkumar}), as in general they do not hold in positive characteristic. We continue with the notation from Section \ref{sec:mainresults}. We denote by $\omega_Y$ the canonical sheaf of a Cohen--Macaulay variety $Y$ and put $\eta=\mc{V}(V^*)$ as in Remark \ref{rem:xi}.

\begin{proposition}\label{prop:duality}
Take $w\in \mc{W}^I$ and put $c=\dim X(w)+\dim V - \dim \ol{BwV}$. If $G\cdot V$ is good then $\mathbf{R}^c q_*  \, \omega_{\ol{BwP}\times_P V} \cong \omega_{\ol{BwV}}$ and  
\[H^i(X(w)_P, \,\, \Sym_d  \eta \oo  \det \eta \oo \omega_{X(w)_P}\,)=0 \mbox{ for all } i\neq c, d\geq 0.\]
\end{proposition}

\begin{proof}
Put $Y=\ol{BwP}\times_P V$ and $Z=\ol{BwV}$. By Theorem \ref{thm:main} we have $\mathbf{R} q_* \mc{O}_Y \cong \mc{O}_Z$. and $Z$ is Cohen--Macaulay (\ref{eq:f-sing}). As $q^{!} \omega_Z \cong \omega_Y[c]$, we obtain using Grothendieck duality \cite[Theorem III.11.1]{hartshorne2}
\[\mathbf{R} q_*  \omega_Y \cong \mathbf{R} q_* \mc{H}\!om_{\mc{O}_Y} (\mc{O}_Y, \omega_Y) \cong \mc{H}\!om_{\mc{O}_Z} (\mc{O}_Z, \omega_Z[-c])\cong \omega_Z[-c].\]
The conclusion follows by the adjunction formula \cite[Proposition II.8.20]{hartshorne}.
\end{proof}

\begin{remark}\label{rem:bott}
When $X(w)=G/P$ and $\cha \kk = 0$, the bundle  $\Sym_d  \eta \oo  \det \eta \oo \omega_{X(w)_P}$ is semi-simple. Thus, using the Borel--Weil--Bott theorem (see \cite[Section 4]{weymanbook} and \cite[Corollary 5.5]{jantzen}) and Serre duality \cite[Corollary 7.7]{hartshorne}, in this case we can deduce from Proposition \ref{prop:duality} that the $L$-dominant weights that appear in $\Sym V \oo \det V$ are either singular or lie in a single Bott chamber (giving cohomology in degree $\dim G\cdot V - \dim V$).
\end{remark}

If we only assume that $V$ is good, one can give an analogous result to Proposition \ref{prop:duality} using normalization as in Remark \ref{rem:normalize}. Along these lines, we give the following version of Griffiths' vanishing theorem \cite{griff} for Schubert varieties in positive characteristic.

\begin{corollary}\label{cor:griff}
Assume $V$ is a good and let $\la \in X(T)_+$ with $\langle \la , \al_i^\vee \rangle = 0$ if and only if $i\in I$ (i.e.\ $\mc{L}(\la)$ is ample on $G/P$). Then
\[H^i(X(w)_P, \,\,  \Sym_d  \eta \oo \det \eta \oo \mc{L}(\la) \oo \omega_{X(w)_P}\,)=0 \mbox{ for all } i>0, d\geq 0, w\in \mc{W}^I.\]
\end{corollary}

\begin{proof}
We put $W'= \Delta_G(\la)\oplus W$, $V' = \kk_{\la} \oplus V$ and consider $q: G \times_P V' \to G\cdot V'$. To conclude by Proposition \ref{prop:duality} in combination with Remark \ref{rem:normalize}, it is enough to show that $q$ is an isomorphism on the open $G\times_P ((\kk_{\la} \setminus\{0\}) \times V)$ (so $q$ is birational). It is known that the map $q_1 : G\times_P \kk_{\la} \to G\cdot \kk_{\la}$ is an isomorphism on the open $G\times_P (\kk_{\la} \setminus\{0\})$ (e.g.\ \cite[Exercise 5.8]{weymanbook}). Further, we have an isomorphism $G\times_P (\kk_{\la} \times W) \cong (G\times_P \kk_{\la}) \times W$ given by $(g,l,w) \mapsto (g,l,gw)$. Composing the latter map with $q_1$ we obtain the result.
\end{proof}

Note that when $\eta$ is a line bundle and $d=0$ (or when $V=0$), the result amounts to the classical Kodaira-type vanishing property for Schubert varieties that can be realized as a consequence of Frobenius splitting or global $F$-regularity \cite{schubsplit}, \cite{schubfreg}, \cite{smith2}. 

\bibliographystyle{alpha}
\bibliography{biblo}

\end{document}